\documentclass[11pt]{amsart}
\usepackage{amsmath,amssymb}
\numberwithin{equation}{section}
 \usepackage{xcolor}

\newtheorem{theorem}{Theorem}[section]
\newtheorem{thm}[theorem]{Theorem}

\newtheorem{lem}[theorem]{Lemma}
\newtheorem{prop}[theorem]{Proposition}
\newtheorem{coro}[theorem]{Corollary}
\newtheorem{corollary}[theorem]{Corollary}

\theoremstyle{definition}

\newtheorem{defi}[theorem]{Definition}

\newtheorem{remark}[theorem]{Remark}
 \newtheorem*{ackn}{Acknowledgements}
 
 \newtheorem*{thmA}{Theorem A} 
 \newtheorem*{thmB}{Theorem B}

 \theoremstyle{plain}
\newtheorem*{namedthm}{\namedthmname}
\newcounter{namedthm}

\makeatletter

 \newcommand{\D}{\mathbb D}
 \newcommand{\R}{\mathbb R}
 
 \newcommand{\C}{\mathbb C}

 \newcommand{\e}{\varepsilon}

 \newcommand{\f}{\varphi}
 
 \newcommand{\p}{\psi}

 \newcommand \la {\lambda}

 \newcommand \psh {{\rm PSH}}
 \newcommand \PSH {{\rm PSH}}

\newcommand{\Ric}{{\rm Ric}}

 \usepackage{hyperref}
\hypersetup{
    unicode=false,        
    pdftoolbar=true,      
    pdfmenubar=true,       
    pdffitwindow=false,     
    pdfstartview={FitH},    
    pdftitle={Uniform estimates for CMAE},    
    pdfauthor={Guedj, Lu},     
    colorlinks=true,       
   linkcolor=blue,          
    citecolor=blue,        
    filecolor=black,      
    urlcolor=blue}



\begin{document}

\title[Quasi-plurisubharmonic envelopes 1]{Quasi-plurisubharmonic envelopes 1: uniform estimates on K\"ahler manifolds}

\author{Vincent Guedj \& Chinh H. Lu}

\address{Institut de Math\'ematiques de Toulouse   \\ Universit\'e de Toulouse \\
118 route de Narbonne \\
31400 Toulouse, France\\}

\email{\href{mailto:vincent.guedj@math.univ-toulouse.fr}{vincent.guedj@math.univ-toulouse.fr}}
\urladdr{\href{https://www.math.univ-toulouse.fr/~guedj}{https://www.math.univ-toulouse.fr/~guedj/}}

\address{Universit\'e Paris-Saclay, CNRS, Laboratoire de Math\'ematiques d'Orsay, 91405, Orsay, France.}

\email{\href{mailto:hoang-chinh.lu@universite-paris-saclay.fr}{hoang-chinh.lu@universite-paris-saclay.fr}}
\urladdr{\href{https://www.imo.universite-paris-saclay.fr/~lu/}{https://www.imo.universite-paris-saclay.fr/~lu/}}
\date{\today}

 \begin{abstract}
 We develop a new   approach to $L^{\infty}$-a priori estimates for degenerate complex Monge-Amp\`ere equations on complex manifolds.
 It only relies on compactness and envelopes properties of quasi-plurisubharmonic functions.
 Our method  allows one to obtain new and efficient proofs of several fundamental
 results in K\"ahler geometry as  we explain in this  article.
 
 In a sequel we shall explain how this approach also applies to  the hermitian setting producing
 new relative a priori bounds, as well as existence results.
 \end{abstract}

 \maketitle

\tableofcontents

\section*{Introduction}

Complex Monge-Amp\`ere equations have been one of the most powerful tools in K\"ahler geometry since 
Yau's solution to the Calabi conjecture \cite{Yau78}.
A notable application is the construction of K\"ahler-Einstein metrics:
given $(X,\omega)$ a compact K\"ahler manifold of complex dimension $n$ and 
$\mu$ an appropriate volume form normalized by $\mu(X)=\int_X \omega^n$,
one seeks for  a solution $\f:X \rightarrow \R$ to
$$
(\omega+dd^c \f)^n=e^{-\la  \f}\mu,
$$
where $d=\partial+\overline{\partial}$, $d^c=i (\partial-\overline{\partial})$
and $\la  \in \R$ is a constant whose sign depends on that of $c_1(X)$.
The metric $\omega_{\f}:=\omega+dd^c \f$ is then K\"ahler-Einstein
as 
$\Ric(\omega_{\f})=\la \omega_{\f}$. 

When $\la \leq 0$, Yau \cite{Yau78} (see also \cite{Aub78} when $\la<0$)
showed the existence of a unique (normalized) solution $\f$ by
establishing a priori estimates along a continuity method, the most delicate one being the uniform
a priori estimate that he established by using  Moser iteration process.

In recent years degenerate complex Monge-Amp\`ere equations have been intensively studied by many authors.
 In relation to the Minimal Model Program, they led to the construction of singular K\"ahler-Einstein metrics
 (see \cite{EGZ09,GZbook,BBEGZ19} and the references therein).
 The main analytical input came here from pluripotential theory which allowed Kolodziej \cite{Kol98}
 to establish uniform a priori estimates 
 when $\mu=fdV_X$ has density in $L^p$ for some $p>1$.

 Using   different methods (Gromov-Hausdorff techniques),
the case $\la>0$
 (Yau-Tian-Donaldson conjecture) 
 has been settled by 
Chen-Donaldson-Sun \cite{CDS15,DonICM18}. Again establishing a uniform a priori estimate in this
context turned out to be the most delicate issue, a key step being obtained
by Donaldson-Sun \cite{DS14} through a refinement of H\"ormander $L^2$-techniques.
An alternative pluripotential variational approach has been developed by Berman-Boucksom-Jonsson in \cite{BBJ15},
based on finite energy classes studied in \cite{GZ07} and variational tools
obtained in \cite{BBGZ13}.
This approach has been pushed one step further by Li-Tian-Wang who have settled the case 
of singular Fano varieties \cite{LTW20}.
 
 \medskip

The main goal of this article is to provide yet another approach for establishing such uniform a priori estimates.
While the pluripotential approach consists in measuring the Monge-Amp\`ere capacity of
sublevel sets $(\f<-t)$, we directly measure the volume of the latter, avoiding
delicate integration by parts. Our approach thus extends with minor modifications
to the hermitian (non K\"ahler) setting, providing several new results
that will be discussed in a companion paper \cite{GL21}:
the hermitian setting introduces several technicalities and new challenges
that might affect the clarity of exposition and could scare the K\"ahler  reader away.

In the whole article we  let thus $X$ denote a compact K\"ahler  manifold of complex dimension $n$.
We fix $\omega$ a closed semi-positive $(1,1)$-form 
which is big,  i.e.
$$
V:=\int_X \omega^n >0.
$$
We let $\PSH(X,\omega)$ denote the set of $\omega$-plurisubharmonic functions: these are
functions $u:X \rightarrow \R \cup \{-\infty\}$ which are locally given as the sum of 
a smooth and a plurisubharmonic function, and such that
$\omega+dd^c u \geq 0$ is a positive current.

 \smallskip
 
Our first main result is 
a brand new proof of
the following   a priori estimate :

\begin{thmA}
{\it 
 Let $\omega$ be semi-positive and big.
   Let $\mu$ be a probability measure  such that
  $\PSH(X,\omega) \subset L^m(\mu)$ for some $m>n$.
Any bounded solution $\f \in\PSH(X,\omega)$  to 
 $V^{-1}(\omega+dd^c \f)^n=\mu$  satisfies a uniform a priori bound
 $$
 {\rm Osc}_X (\f) \leq  T_{\mu}
 $$
 for some uniform constant $T_{\mu}=T(A_m(\mu))$ which depends on an upper bound on
 $$
A_m({\mu}):= \sup \left\{ \int_X (-\p)^m d\mu, \; \p \in \PSH(X,\omega) \text{ with } \sup_X \p=0 \right\}.
 $$
 }
\end{thmA}

H\"older inequality shows that this result covers the case when $\mu=fdV_X$ is absolutely continuous with respect to Lebesgue measure,
with density $f$ belonging to $L^p$, $p>1$, or to an appropriate Orlicz class, as we explain in
Section \ref{sec:abscont}.

A crucial particular case of this estimate is due to   Kolodziej \cite{Kol98}. 
Other important special cases have been previously obtained in \cite{EGZ09,EGZ08,DP10}.
 Our new method covers all these settings at once, it also permits to recover the main estimates 
of \cite{BEGZ10} (big cohomology classes) and \cite{DnGG20} (collapsing families)
as we explain in Sections \ref{sec:big} and \ref{sec:families}.
A slight refinement of our technique allows one to establish an important stability estimate
(see Theorem \ref{thm:stability}).

\smallskip

There are several  geometric situations when one can not expect the Monge-Amp\`ere
potential $\f$ to be  globally bounded. 
 We next consider the   equation 
 \begin{equation*}  
V^{-1}  (\omega+dd^c \f)^n=   fdV_X,
 \end{equation*} 
where  the density $f \in L^1(X)$ does not belong to any  
good Orlicz class.
Since the  measure $\mu=fdV_X$ is non pluripolar, there exists
a unique finite energy solution $\f$ (see \cite{GZbook}, \cite{Din09}).
It is crucial to understand its locally bounded locus.
 
As $\omega$  is a   semi-positive and big $(1,1)$ form, 
we can find
$\rho$ an $\omega$-psh function with analytic singularities such that
$\omega+dd^c \rho \geq \delta \omega_X$ is a K\"ahler current
(see \cite[Theorem 0.5]{DP04}).
For  $\p$ quasi-psh and  $c>0$, we set
$$
E_c(\p):=\{ x \in X , \; \nu(\p,x) \geq c \},
$$
where $\nu(\p,x)$ denotes the  Lelong number of $\p$ at $x$. 
A celebrated theorem of Siu ensures that for any $c>0$, the
set $E_c(\p)$ is a closed analytic subset of $X$.

Our second main result provides
the following a priori estimate,
which extends  
a result of DiNezza-Lu \cite{DnL17}:

\begin{thmB}
{\it 
Assume   $f=g e^{-\p}$, where $0 \leq g \in L^p(dV_X)$, $p>1$,  and  $\p$ is a quasi-psh function.
Then  there exists a unique $\f \in \mathcal{E}(X,\omega)$ such that
\begin{itemize}
\item $\alpha (\p+\rho)-\beta \leq \f \leq 0$ with $\sup_X \f=0$;
\item $\f$ is locally bounded in the open set $\Omega:= X \setminus \{\rho=-\infty\} \cup E_{\frac{1}{q}}(\p)$;
\item $ V^{-1}(\omega+dd^c \f)^n=   fdV_X 
 \; \;  \text{ in } \; \;
\Omega$,
\end{itemize}
where $\alpha,\beta>0$ depend on  an upper bound for $||g||_{L^p}$
and $\frac{1}{p}+\frac{1}{q}=1$.
}
\end{thmB}

Again the proof we provide is direct,
 and can be extended to the hermitian setting
(see \cite{GL21}). We finally show in Section \ref{sec:local} how the same arguments
can be applied to efficiently solve the Dirichlet problem in pseudoconvex domains.

\medskip

\noindent {\it Comparison with other works.}
Yau's proof of his famous $L^{\infty}$-a priori estimate \cite{Yau78} goes through a   Moser iteration process.
Although Yau could deal with some singularities, the method does not apply when the right hand side is too degenerate
(see however \cite{Cao85,Tos10} for   further applications of Yau's method).

An important generalization of Yau's estimate has been provided by Kolodziej \cite{Kol98} using pluripotential techniques.
These  have been further generalized in \cite{EGZ09,EGZ08,DP10,BEGZ10}
in order to deal with less positive or collapsing families of cohomology classes on K\"ahler manifolds.
As this   approach relies
on delicate integration by parts, it is difficult to extend to the hermitian setting.

Blocki has provided a different approach in \cite{Blo05b} based on the Alexandroff-Bakelman-Pucci maximum principle
and a local stability estimate due to Cheng-Yau ($L^2$-case) and Kolodziej ($L^p$-case).
This has been pushed further by Szekelehydi in \cite{Szek18}. 
It requires the reference form $\omega$ to be strictly positive.

A PDE proof of the $L^{\infty}$-estimate has been very recently provided by 
Guo-Phong-Tong \cite{GPT21} using an auxiliary 
Monge-Amp\`ere equation, inspired by the recent breakthrough
by Chen-Cheng on constant scalar curvature metrics \cite{CC21}.


Our approach consists in showing that the   sublevel set $(\f<-t)$ becomes the empty set in finite time
by directly measuring its $\mu$-size. 
It only
uses weak compactness  of normalized $\omega$-plurisubharmonic functions
and basic properties of quasi-psh envelopes, allowing us to deal
with semi-positive forms.

\begin{ackn} 
This work has benefited from State aid managed
  by the ANR projects PARAPLUI
 and  ANR-11-LABX-0040. We thank Ahmed Zeriahi for carefully reading the paper and giving numerous useful comments. 
\end{ackn}


\section{Quasi-plurisubharmonic envelopes}

In the whole article we let $X$ denote a compact K\"ahler manifold of complex dimension $n \geq 1$.
We fix $\omega$ a smooth closed real $(1,1)$-form on $X$.

\subsection{Monge-Amp\`ere operators}

 \subsubsection{Quasi-plurisubharmonic functions}

  A function is quasi-plurisub\-harmonic if it is locally given as the sum of  a smooth and a psh function.   
 Quasi-psh functions
$\f:X \rightarrow \R \cup \{-\infty\}$ satisfying
$
\omega_{\f}:=\omega+dd^c \f \geq 0
$
in the weak sense of currents are called $\omega$-plurisubharmonic ($\omega$-psh for short).

\begin{defi}
We let $\PSH(X,\omega)$ denote the set of all $\omega$-plurisubharmonic functions which are not identically $-\infty$.  
\end{defi}

Constant functions are $\omega$-psh functions if (and only if) $\omega$ is semi-positive.
A ${\mathcal C}^2$-smooth function $u$ has bounded Hessian, hence $\e u$ is
$\omega$-psh if $0<\e$ is small enough and $\omega$ is  positive.
It is useful to consider as well the case when $\omega$ is not necessarily positive,
in order to study {\it big} cohomology classes (see section \ref{sec:big}).
 
\begin{defi}
A semi-positive closed $(1,1)$-form $\omega$ is big if
$V_{\omega}:=\int_X \omega^n>0$.
\end{defi}

The set $\PSH(X,\omega)$ is a closed subset of $L^1(X)$, 
for the $L^1$-topology.
Subsets of  $\omega$-psh functions enjoy strong compactness and integrability properties,
we mention notably the following: for any fixed $r \geq 1$, 
\begin{itemize}
\item $\PSH(X,\omega) \subset L^r(X)$;
 the induced $L^r$-topologies are equivalent;
\item $\PSH_A(X,\omega):=\{ u \in\PSH(X,\omega), \, -A \leq \sup_X u \leq 0 \}$ is compact in  $L^r$.
\end{itemize}
We refer the reader to \cite{Dem12,GZbook} for further basic properties of $\omega$-psh functions.

\subsubsection{Monge-Amp\`ere measure}

The complex Monge-Amp\`ere measure 
$$
 (\omega+dd^c u)^n=\omega_u^n
$$
 is well-defined for any
$\omega$-psh function $u$ which is {\it bounded}, as follows from Bedford-Taylor theory 
(see \cite{BT82} for the local theory, and \cite{GZbook} for the compact K\"ahler context).
 It also makes sense in the ample locus of
a big cohomology class \cite{BEGZ10}, as we shall briefly discuss in section \ref{sec:big}.

\smallskip

The mixed Monge-Amp\`ere measures 
$(\omega+dd^c u)^j \wedge (\omega+dd^c v)^{n-j}$ are also
well defined for any $0 \leq j \leq n$, and any bounded $\omega$-psh functions $u,v$.
We note for later use the following classical inequality: 

\begin{lem} \label{lem:Demkey}
Let $\f,\p$ be bounded $\omega$-psh functions such that $\f \leq \p$, then
$$
 1_{\{\p = \f \}} (\omega+dd^c \f)^j \wedge (\omega+dd^c \p)^{n-j} \leq 1_{\{\p = \f \}} (\omega+dd^c \p)^n,
$$
for all $1 \leq j \leq n$. 
\end{lem}

\begin{proof}
To simplify notations we just  treat the case $j=n$.
It follows from Bedford-Taylor theory \cite{BT82} that for any bounded $\omega$-psh functions
$\f,\p$,
$$
1_{\{\p \leq \f \}}  \omega_{\f}^n+1_{\{\p>\f\}}  \omega_{\p}^n \leq (\omega+dd^c \max(\f,\p))^n
$$
When $\f \leq \p$ we infer
$
1_{\{\p = \f \}} \omega_{\f}^n \leq 1_{\{\p = \f \}} \omega_{\p}^n .
$
\end{proof}

 We shall also need the 
 following (see \cite[Proposition 10.11]{GZbook}):

  \begin{prop} \label{cor:domination}
  [Domination principle]
 If $u,v$ are bounded $\omega$-psh functions such that $u\geq v$ a.e. with respect to  $\omega_u^n$. Then $u\geq v$. 
 \end{prop}


\subsection{Envelopes} \label{sec:envelopes}

Upper envelopes of (pluri)subharmonic functions are classical objects in Potential Theory. 
They  were considered   by  Bedford and Taylor to solve the Dirichlet problem for 
the complex Monge-Amp\`ere equation in strictly pseudo-convex domains \cite{BT76}.
We consider here envelopes of $\omega$-psh functions.

\subsubsection{Basic properties}


\begin{defi} \label{def:usual}
Given a Lebesgue measurable function $h:X \rightarrow \R$, we define the $\omega$-psh envelope of $h$ by
$$
P_{\omega}(h) := \left(\sup \{ u \in \psh (X,\omega) ; u \leq h  \, \, \text{in} \, \, X\}\right)^*,
$$
where the star means that we take the upper semi-continuous regularization. 
\end{defi}


\noindent The following  is a combination of \cite[Propositions  2.2 and 2.5, Lemma 2.3]{GLZ19}:

 \begin{prop} \label{pro:orthog}
If $h$ is bounded from below and quasi-continuous, then
\begin{itemize}
\item $P_{\omega}(h)$ is a bounded $\omega$-plurisubharmonic function;
\item $P_{\omega}(h) \leq h$ in $X \setminus P$, where $P$ is pluripolar;
\item $(\omega+dd^c P_{\omega}(h))^n$ is concentrated on the contact $\{ P_{\omega}(h)=h\}$.
\end{itemize}
\end{prop}

Recall that a function $h$ is quasi-continuous if for any $\e>0$,
there exists an open set $G$ of capacity smaller than $\e$ such that
$h$ is continuous in $X \setminus G$. Quasi-psh functions are quasi-continuous (see \cite{BT82}),
as well as differences of the latter: we shall use this fact during the proof of Theorem \ref{thm:relative}.

 When $h$ is ${\mathcal C}^{1,1}$-smooth,   so is $P_{\omega}(h)$ 
 \cite{Ber19,CZ19} and one can further has
 \begin{equation} \label{eq:MAenv}
(\omega+dd^c P_{\omega}(h))^n=1_{\{ P_{\omega}(h)=h\}} (\omega+dd^c h)^n.
 \end{equation}

 \subsubsection{A key lemma}
 
 The following is a key technical tool to our new approach:
 
 \begin{lem} \label{lem:GLkey}
 Fix $\chi:\R^- \rightarrow \R^-$ a concave increasing function such that $\chi'(0) \geq 1$. 
  Let $\f,\phi$ be bounded $\omega$-psh functions with $\f \leq \phi$. If $\p=\phi+\chi \circ (\f-\phi)$
 then
$$
  (\omega+dd^c P_{\omega}(\p))^n\leq 1_{\{ P_{\omega}(\p)=\p\}} (\chi' \circ (\f-\phi))^n (\omega+dd^c \f)^n.
 $$
 \end{lem}
 
 \begin{proof}
 Using that $\chi'' \leq 0$ and $\chi' \geq 1$, we observe that 
 \begin{eqnarray*}
\omega+dd^c \p &= &  \omega_{\phi} + \chi' \circ (\f-\phi) (\omega_\f -\omega_{\phi}) +\chi'' \circ (\f-\phi) d(\f-\phi) \wedge d^c( \f -\phi) \\
&\leq & \chi' \circ (\f-\phi) \omega_\f +[1-\chi' \circ (\f-\phi)] \omega_{\phi} 
\leq  \chi' \circ (\f-\phi) \omega_\f.
 \end{eqnarray*}

 When $\f,\phi$ and $\chi$ are ${\mathcal C}^{1,1}$-smooth, we can invoke \eqref{eq:MAenv} to conclude that
 $$
   (\omega+dd^c P_{\omega}(\p))^n = 1_{\{ P_{\omega}(\p)=\p\}} \omega_{\p}^n     \leq  1_{\{ P_{\omega}(\p)=\p\}} (\chi' \circ (\f-\phi))^n \omega_{\f}^n.
 $$
   The last inequality follows   from   $\omega+dd^c \p  \leq \chi' \circ (\f-\phi) \omega_\f$
 and the fact that $\p$ is $\omega$-psh on  $\{ P_{\omega}(\p)=\p\}$,
 where these inequalities can be interpreted pointwise.
 
 When these functions are less regular we take a different route.
 We set $\tau=\chi^{-1}:\R^- \rightarrow \R^-$. This is a convex increasing function such that 
 $\tau'=(\chi' \circ \tau)^{-1} \leq 1$.
 Set $\rho=P_{\omega}(\p)-\phi$.
 The function $v=\phi+\tau \circ (P_{\omega}(\p)-\phi)$ is $\omega$-psh with
 \begin{eqnarray*}
\omega+  dd^c v & =& \omega_{\phi}+\tau'' \circ \rho  \, d \rho \wedge d^c \rho
  +\tau' \circ \rho \,  dd^c (P_{\omega}(\p)-\phi) \\
  &\geq& [1- \tau' \circ \rho] \omega_{\phi}+\tau' \circ \rho \,  (\omega+dd^c  P_{\omega}(\p))  \\
  &\geq & \tau' \circ \rho \,  (\omega+dd^c  P_{\omega}(\p)).
 \end{eqnarray*}
Thus
$
\omega_{P_{\omega}(\p)}^n \leq  1_{\{ P_{\omega}(\p)=\p\}}  (\tau' \circ (P_{\omega}(\p)-\phi))^{-n} \omega_v^n.
$
On 
$\{ P_{\omega}(\p)=\p\}$  we get
$$
\tau' \circ (P_{\omega}(\p)-\phi) = \tau' \circ (\p-\phi)=[\chi' \circ (\f-\phi)]^{-1}.
$$
Now $v \leq \phi+ \tau \circ (\p-\phi)=\f$ on $X$, with equality 
on   the contact set $\{ P_{\omega}(\p)=\p\}$. 
It follows therefore from Lemma \ref{lem:Demkey} that
  $\omega_v^n \leq \omega_{\f}^n$ on  $\{ P_{\omega}(\p)=\p\}$.
 \end{proof}

\section{Global $L^{\infty}$ bounds}

In this section we prove Theorem A, as well as a stability estimate.

\subsection{Measures which integrate quasi-plurisubharmonic functions}

   \begin{theorem}  \label{thm:uniform1}
   Let $\omega$ be semi-positive and big.
   Let $\mu$ be a probability measure  such that
  $\PSH(X,\omega) \subset L^m(\mu)$ for some $m>n$.
Any solution $\f \in\PSH(X,\omega) \cap L^{\infty}(X)$  to 
 $V^{-1}(\omega+dd^c \f)^n=\mu$  satisfies
 $$
 {\rm Osc}_X (\f) \leq  T_{\mu}
 $$
 for some uniform constant $T_{\mu}=T(A_m(\mu))$ which depends on  an upper bound on
 $$
A_m({\mu}):= \sup \left\{ \left(\int_X (-\p)^m d\mu\right)^{\frac{1}{m}}, \; \p \in \PSH(X,\omega) \text{ with } \sup_X \p=0 \right\}.
 $$
 \end{theorem}

  
 
 Let us stress that this   result is not new: it can be derived
 from the celebrated a priori estimate of Kolodziej \cite{Kol98},
 together with its extensions \cite{EGZ09,EGZ08,DP10}.
We provide here 
 an elementary proof that does not use the theory of Monge-Amp\`ere capacities,
 and merely relies on 
 the compactness properties of sup-normalized $\omega$-psh functions
 and Lemma \ref{lem:GLkey}.

 \begin{proof}
 Shifting by an additive constant, we normalize $\f$ by $\sup_X \f=0$. Set 
 \[
 T_{\max}:= \sup \{t>0 \; : \; \mu(\f <-t) >0\}. 
 \]
 Our goal is to establish a precise bound on $T_{max}$.
    By definition,  $-T_{max} \leq \f$ almost everywhere with respect to $\mu$, 
    hence everywhere by the domination principle (Proposition \ref{cor:domination}), providing the desired a priori bound
    ${\rm Osc}_X (\f) \leq T_{max}$.

%

 We let $\chi:\R^- \rightarrow \R^-$ denote a {\it concave} increasing function
 such that $\chi(0)=0$ and $\chi'(0) = 1$. 
 We set $\p=\chi \circ \f$, $u=P_{\omega}(\p)$ and observe that
$$
\omega+dd^c \p = \chi' \circ \f \, \omega_\f+[1-\chi' \circ \f] \, \omega+\chi'' \circ \f \, d\f \wedge d^c \f 
\leq  \chi' \circ \f \, \omega_\f.
$$
It follows from Lemma \ref{lem:GLkey} that
 $$
MA(u):=\frac{1}{V} (\omega+dd^c u)^n  
 \leq 1_{\{u=\p\}}  (\chi' \circ \f)^n \mu.
 $$

 \smallskip
 
\noindent  {\it Controlling the energy of $u$}.
We fix $\e>0$ so  that 
$n<n+3\e =m.$
 The concavity of $\chi$ and the normalization $\chi(0)=0$ yields $|\chi(t)| \leq |t| \chi'(t)$. 
 Since $u=\chi \circ \f$ on  
 the contact set $\{P_{\omega}(\p)=\p\}$,  H\"older inequality yields
 \begin{eqnarray*}
 \int_X (-u)^{\e} MA(u)
 &\leq &  \int_X (-\chi \circ \f)^{\e} (\chi' \circ \f)^n d\mu   
 \leq   \int_X (- \f)^{\e} (\chi' \circ \f)^{n+\e} d\mu   \\
 & \leq &  \left( \int_X (-\f)^{n+2\e} d\mu \right)^{\frac{\e}{n+2\e}} 
 \left( \int_X ( \chi' \circ \f)^{n+2\e} d\mu \right)^{\frac{n+\e}{n+2\e}} \\
 &=& A_m({\mu})^{\e} \left( \int_X ( \chi' \circ \f)^{n+2\e} d\mu \right)^{\frac{n+\e}{n+2\e}} 
 \end{eqnarray*}
  using that $\f$ belongs to the set of $\omega$-psh functions $v$  normalized by $\sup_X v=0$ 
 which is compact in $L^{n+2\e}(\mu)$, and observing
 that $A_{n+2\e}({\mu}) \leq A_m({\mu})$.
 
  \smallskip
 
\noindent  {\it Controlling the norms $||u||_{L^m}$}.
We are going to choose below the weight $\chi$ in such a way that
$\int_X ( \chi' \circ \f)^{n+2\e} d\mu =B \leq 2$ is a finite constant under control.
This provides a uniform lower bound on $\sup_X u$ as we now explain:
indeed
$$
0 \leq  (-\sup_X u)^{\e}   =(-\sup_X u)^{\e} \int_X  MA(u)  
\leq \int_X (-u)^{\e} MA(u)  \leq 2 A_m({\mu})^{\e}
$$
 yields
 $
  -2^{\frac{1}{\e}} A_m({\mu}) \leq \sup_X u \leq 0.
  $
We infer that  $u$ belongs to a compact set of $\omega$-psh functions, hence its norm
 $||u||_{L^m(\mu)}$ is under control with
 $$
 ||u||_{L^m(\mu)} \leq A_m(\mu)+2^{\frac{1}{\e}} A_m({\mu}) \leq [1+2^{\frac{1}{\e}}] A_m({\mu}).
 $$
 Since $u \leq \chi \circ \f \leq 0$
we infer $||\chi \circ \f||_{L^m} \leq ||u||_{L^m}$.   Chebyshev inequality thus yields
\begin{equation} \label{eq:clef}
 {\mu }(\f<-t) \leq \frac{\tilde{A}}{|\chi|^m(-t)},
 \; \; \text{ where } \; \;
 \tilde{A}=[1+2^{\frac{1}{\e}}] A_m({\mu}).
\end{equation}

   \smallskip
 
\noindent  {\it Choice of $\chi$}.
 Lebesgue's formula ensures that if $g: \R^+ \rightarrow \R^+$  is an increasing function such that $g(0)=1$, then
$$
\int_X g \circ (-\f) d\mu = \mu(X) + \int_0^{T_{\max}} g'(t) {\mu }(\f<-t) dt.
$$
Fix $0<T_0<T_{\max}$. 
Setting $g(t)=[\chi' (-t)]^{n+2\e}$  we define $\chi$ by imposing $\chi(0)=0$, $\chi'(0)=1$, and 
$$
g'(t)=
\begin{cases}
	\dfrac{1}{(1+t)^2 {\mu }(\f<-t)}, \; \text{if} \; t\leq T_0\\
	\; 
	\\	\frac{1}{(1+t)^2} \; \; \; \; \text{ if}\;  t >  T_0
\end{cases}.
$$
This choice guarantees that $\chi: \mathbb{R}^- \rightarrow \mathbb{R}^-$ is concave increasing with $\chi' \geq 1$, and
$$
\int_X ( \chi' \circ \f)^{n+2\e} d\mu \leq  \mu(X) + \int_0^{+\infty} \frac{dt}{(1+t)^2} =2.
$$

\smallskip

 \noindent  {\it Conclusion}.  We set $h(t)=-\chi(-t)$ and work with the positive counterpart of $\chi$. Note that $h(0)=0$ and $h'(t)=[g(t)]^{\frac{1}{n+2\e}}$ is positive increasing, hence
$h$ is convex.
Observe also that $g(t) \geq g(0)=1$ hence $h'(t) =[g(t)]^{\frac{1}{n+2\e}} \geq 1$ yields
\begin{equation} \label{eq:minh(1)}
h(1)=\int_0^1 h'(s) ds \geq 1.
\end{equation}

Together with \eqref{eq:clef} our choice of $\chi$ yields, for all $t\in [0,T_0]$, 
$$
\frac{1}{(1+t)^2g'(t)}={\mu }(\f<-t) \leq \frac{\tilde{A}}{h^m(t)}.
$$
For $t\in [0,T_0]$, this reads
$$
h^m(t) \leq \tilde{A} (1+t)^2 g'(t)=(n+2\e) \tilde{A} (1+t)^2 h''(t)  (h')^{n+2\e-1}(t). 
$$
Multiplying  by $h'$,  integrating between $0$ and $t$, 
we infer that for all $t \in [0,T_0]$,
\begin{eqnarray*}
\frac{h^{m+1}(t)}{m+1}
&\leq&   (n+2\e)  \tilde{A}   \int_0^t (1+s)^2 h''(s)  (h')^{n+2\e}(s)\\
&\leq & \frac{(n+2\e) \tilde{A} (t+1)^2 }{n+2\e+1} \left ((h')^{n+2\e+1}(t) - 1 \right)  \\
&\leq&    \tilde{A}    (1+t)^2   (h')^{n+2\e +1}(t)  .
\end{eqnarray*}
Recall that    $m=n+3 \e$ so that 
$
\alpha:=m+1> \beta:= n+2\e +1>2.
$
The previous inequality then reads
$$
 (1+t)^{-\frac{2}{\beta}} \leq C  {h'(t)}{h(t)^{-\frac{\alpha}{\beta}}},
$$
for some uniform constant $C$ depending on $n,m,\tilde{A}$.
Since $\alpha>\beta>2$ and $h(1) \geq 1$, 
integrating the above inequality between $1$ and $T_0$ we obtain 
$
T_0 \leq C', 
$
for some uniform constant $C'$ depending on $C,\alpha,\beta$. 
Since $T_0$ was chosen arbitrarily in $(0,T_{\max})$ the result follows. 
 \end{proof}

 \subsection{Absolutely continuous measures}  \label{sec:abscont}
  
Assume $\mu=f dV_X$ is absolutely continuous 
 with respect to 
 a volume form $dV_X$,
 with density $0 \leq f \in L^p(dV_X)$ for some $p>1$.
 Since $\PSH(X,\omega) \subset L^r(dV_X)$ for any $1 \leq r <+\infty$, we obtain
 $$
 \int_X |u|^m d\mu \leq ||f||_{L^p(dV_X)} \cdot  \left( \int_X |u|^{qm} dV_X \right)^{1/q},
 $$
  for all $u \in\PSH(X,\omega)$,
 where $1/p+1/q=1$, so that $\PSH(X,\omega) \subset L^m(d\mu)$ for all $m \geq 1$.
 Thus Theorem \ref{thm:uniform1} applies to this type of measures, providing a new proof
 of the celebrated a priori estimate of Kolodziej \cite{Kol98} (see also \cite{EGZ09}).
 
 \smallskip
 
 As in \cite{Kol98} our technique also covers the case of more general densities as we briefly indicate.
 Let $w:\R^+ \rightarrow \R^+$ be a convex increasing weight.
A measurable function $f$ belongs to the Orlicz class $L^{w}(dV_X)$ if there exists $\alpha>0$ such that
 $$
 \int_X w(\alpha |f|) dV_X <+\infty.
 $$
 The Luxembourg norm of $f$ is defined as
 $$
 ||f||_{w}:=\inf \{ r>0, \;  \int_X w( |f|/r) dV_X  \leq 1 \};
 $$
 it turns $L^{w}(dV_X)$ into a Banach space. 
 
 If $w^*$ denotes the conjugate convex weight of $w$ (its Legendre transform), 
 H\"older-Young inequality ensures that for all  measurable functions $f,g$,
 $$
 \int_X |fg| dV_X \leq 2  ||f||_{w}  ||g||_{w^*}.
 $$
 We refer the reader to \cite{RR} for more information on Orlicz classes.
 
 \smallskip
 
 Theorem \ref{thm:uniform1} thus allows to reprove 
 \cite[Theorem 2.5.2]{Kol98}:

 \begin{corollary} 
 Let $\mu=fdV_X$ be a probability measure.
 Let $w:\R^+ \rightarrow \R^+$ be a convex increasing weight  
  that grows at infinity at least like $t (\log t)^{m}$ with $m>n$.
    If  $f$ belongs to the Orlicz class $L^{w}$ then
 any solution $\f \in \PSH(X,\omega) \cap L^{\infty}(X)$  to 
 $V^{-1}(\omega+dd^c \f)^n=\mu$ satisfies
 $$
 {\rm Osc}_X (\f) \leq  T_{\mu}
 $$
 for some uniform constant $T_{\mu} \in \R^+$.
 \end{corollary}
 
\begin{proof}
While this was not required for the case of $L^p$ densities, we need here to invoke
Skoda's uniform integrability result (see \cite[Theorem 8.11]{GZbook}):
there exists $\alpha>0$ and $C=C(\alpha,M)>0$ such that
$$
\sup \left\{ \int_X e^{2 \alpha |u|} dV_X, \; u \in \PSH(X,\omega) \text{ and } -M \leq \sup_X u \leq 0 \right\} \leq C.
$$

The reader will check that, as $s \rightarrow +\infty$, the 
conjugate weight
$w^*(s)$ grows like 
$$
w^*(s) \sim s^{1-\frac{1}{m}} \exp(s^{\frac{1}{m}}) \leq  \exp(2 s^{\frac{1}{m}}).
$$
It follows therefore from Young inequality that any $\omega$-psh function
$u$ satisfies
$$
\alpha^m \int_X |u|^{m} d\mu \leq \int_X w \circ f dV_X + \int_X \exp( 2 \alpha |u|) dV_X <+\infty.
$$
Thus $\PSH(X,\omega) \subset L^m(\mu)$ and the conclusion follows from Theorem \ref{thm:uniform1}.
\end{proof}
 
 One can slightly improve the assumption on the density 
 as in  \cite[Theorem 2.5.2]{Kol98}, we leave the technical details to the interested reader.

  \begin{remark}
  It follows from the 
  Chern-Levine-Nirenberg inequality that
if $\mu=(\omega+dd^c \f)^n$ is the Monge-Amp\`ere measure of 
a bounded $\omega$-psh function, then
$\PSH(X,\omega) \subset L^1(\mu)$.
If $n=1$ this condition   is 
 equivalent to $\mu$ having bounded potential (see \cite[Lemma 3.2]{DnGL20}).
 Note however that when $n \geq 2$,
 \begin{itemize}
 \item the condition  $\PSH(X,\omega) \subset L^n(\mu)$,  $\mu=(\omega+dd^c \f)^n$,
 is not sufficient to guarantee that the $\omega$-psh function $\f$ is bounded ;
 \item  
one cannot improve the C-L-N inequality:
there are examples
of  Monge-Amp\`ere measures with bounded potential
and $\PSH(X,\omega) \not\subset L^{1+\e}(\mu)$.
 \end{itemize}
 \end{remark}
 
 

\subsection{Stability estimate}

 We  now
 establish the following stability estimate,
 which can be seen as a refinement of \cite[Proposition 5.2]{GZ12}.

   \begin{theorem}  \label{thm:stability}
   Let $\omega,\mu$ be as in Theorem \ref{thm:uniform1}. 
Let $\f \in\PSH(X,\omega) \cap L^{\infty}(X)$  be such that $\sup_X \varphi=0$ and
 $V^{-1}(\omega+dd^c \f)^n=\mu$.  Then 
 $$
\sup_X (\phi-\varphi)_+ \leq  T \left (\int_X (\phi-\varphi)_+d\mu \right)^{\tau},
 $$
for any $\phi \in \PSH(X,\omega)\cap L^{\infty}(X)$, 
where 
$\tau=\tau(n,m)>0$ and 
$$
T=T(\mu,\|\phi\|_{L^{\infty}})
$$
 is a uniform constant which depends on
 an upper bound on $||\phi ||_{L^{\infty}}$ and  
 $$
A_m({\mu}):= \sup \left\{ \left( \int_X (-\p)^m d\mu \right)^{\frac{1}{m}}, \; \p \in \PSH(X,\omega) \text{ with } \sup_X \p=0 \right\}.
 $$
 \end{theorem}

 \begin{proof}
 	Replacing $\phi$ by $\max(\varphi,\phi)$, we can assume that $\varphi\leq \phi$. Define
 \[
 T_{\max}:= \sup \{t>0 \; : \; \mu(\f <\phi-t) >0\}. 
 \]
  It follows from Theorem \ref{thm:uniform1} that $T_{\max}$ is uniformly controlled by $\mu$ and $||\phi||_{L^{\infty}}$.

 We let $\chi:\R^- \rightarrow \R^-$ denote a {\it concave} increasing function
 such that $\chi(0)=0$ and $\chi'(0) = 1$. 
 We set $\p= \phi+ \chi \circ (\f-\phi)$, $u=P(\p)$ and observe that
\begin{eqnarray*}
\omega+dd^c \p &= &  \omega_{\phi} + \chi' \circ (\f-\phi) (\omega_\f -\omega_{\phi}) +\chi'' \circ (\f-\phi) d(\f-\phi) \wedge d^c( \f -\phi) \\
&\leq & \chi' \circ (\f-\phi) \omega_\f.
\end{eqnarray*}
It follows from Lemma \ref{lem:GLkey} that
 \[
 MA(u) := \frac{V_{\omega}}{(\omega+dd^c u)^n} \leq   1_{\{u=\p\}} (\chi' \circ (\f-\phi))^n \mu.
 \]

 \smallskip
 
\noindent  {\it Controlling the energy of $u$}.
We fix $0<a<b<c<2c<\e$ so small that 
$$
q:=  \frac{(\varepsilon-a)(n+b)}{b-a}<m=n+\e.
$$
 The concavity of $\chi$ and the normalization $\chi(0)=0$ yields $|\chi(t)| \leq |t| \chi'(t)$. 
 Since $u=\phi +\chi \circ (\f-\phi)$ on the support of $(\omega+dd^c u)^n$
 and $\PSH(X,\omega) \subset L^{n+2c}(\mu)$,
 H\"older inequality yields
 \begin{flalign*}
 0 \leq \int_X (-u+\phi)^{c} & MA(u) \leq   \int_X (-\chi \circ (\f-\phi))^{c} (\chi' \circ (\f-\phi))^n d\mu   \\
 & \leq  \int_X (- \f+\phi)^{c} (\chi' \circ (\f-\phi))^{n+c} d\mu   \\
 & \leq   \left( \int_X (-\f+\phi)^{n+2c} d\mu \right)^{\frac{c}{n+2c}} 
 \left( \int_X ( \chi' \circ (\f-\phi))^{n+2c} d\mu \right)^{\frac{n+c}{n+2c}} \\
 & \leq A_m(\mu)^c \left( \int_X ( \chi' \circ (\f-\phi))^{n+2c} d\mu \right)^{\frac{n+c}{n+2c}}.
 \end{flalign*}
 
  \smallskip
 
\noindent  {\it Controlling the norms $||u||_{L^m}$}.
We 
choose    $\chi$ below s.t.
$\int_X ( \chi' \circ (\f-\phi))^{n+2c} d\mu  \leq B$ is   under control.
This provides a uniform lower bound on $\sup_X u$.
Indeed  our normalizations yield $\chi(t) \leq t$ hence
$
u \leq \phi+\chi (\f-\phi) \leq \f \leq 0,
$
while
$$
0 \leq  (-\sup_X (u-\phi))^{c}  
\leq \int_X (-u+\phi)^{c} MA(u)  
\leq A_m(\mu)^c B^{\frac{n+c}{n+2c}} 
$$
yields a lower bound on $ \sup_X(u-\phi)$.
Now $u=u-\phi+\phi \geq u-\phi+\inf_X \phi$, so 
$\sup_X u \geq \sup_X(u-\phi)+\inf_X \phi \geq -A_m(\mu) B^{\frac{n+c}{c(n+2c)}} +\inf_X \phi$.

Thus  $u$ belongs to a compact set of $\omega$-psh functions:
 its norm $||u||_{L^q(\mu)}$ is under control  for any $q \leq m$.
 Since $u -\phi \leq \chi \circ (\f -\phi)\leq 0$,
 H\"older inequality yields
\begin{flalign}
	\int_X & |\chi \circ (\f-\phi)|^m d\mu   \leq \int_X |\chi\circ (\f-\phi)|^{n+a}  (\phi-u)^{\varepsilon-a}d\mu\nonumber \\
	& \leq \left ( \int_X  |\chi\circ (\varphi-\phi)|^{n+b}d\mu \right)^{\frac{n+a}{n+b}} \left ( \int_X  (\phi-u)^{q}d\mu \right)^{\frac{b-a}{n+b}}\nonumber \\
	& \leq C_{\mu}' \left ( \int_X  |(\phi-\varphi)\chi'\circ (\varphi-\phi)|^{n+b}d\mu \right)^{\frac{n+a}{n+b}}\nonumber \\
	& \leq C_{\mu}' \left ( \int_X  (\phi-\varphi)^{\frac{(n+c)(n+b)}{c-b}}d\mu \right)^{\frac{(c-b)(n+a)}{(n+c)(n+b)}} \left ( \int_X  | \chi'\circ (\varphi-\phi)|^{n+c}d\mu \right)^{\frac{n+a}{n+c}}\nonumber \\
	& \leq C_1 B^{\frac{n+a}{n+c}}  \left ( \int_X (\phi-\varphi) d\mu \right )^{\gamma}  
	=:\tilde{A},  \label{eq: stability1}
\end{flalign}
where $\gamma= \dfrac{(c-b)(n+a)}{(n+c)(n+b)}$, and $C_1$ depends on $C_{\mu}$, 
$||\varphi||_{L^{\infty}}$ and $||\phi||_{L^{\infty}}$.

\smallskip

 It follows therefore from Chebyshev inequality that
\begin{equation} \label{eq:clef bis}
 {\mu }(\f<\phi-t) \leq \frac{\tilde{A}}{|\chi|^m(-t)}.
\end{equation}

   \smallskip
 
\noindent  {\it Choice of $\chi$}. 
Fix $T_0 \in (0,T_{\max})$. 
We set $g(t)=[\chi' (-t)]^{n+2c}$  and define $\chi$ by imposing $\chi(0)=0$, $\chi'(0)=1$,   and
$$
g'(t)=
\begin{cases}
	\dfrac{1}{{\mu }(\f<\phi-t)}, \; \text{if} \; t\leq T_0\\
	\; 
	\\	1 \; \; \; \; \text{ if}\;  t >  T_0
\end{cases}.
$$

This choice guarantees that 
$$
\int_X ( \chi' \circ (\f-\phi))^{n+2c} d\mu \leq \mu(X) + \int_0^{T_{\max}} dt =1+T_{\max}.
$$
It follows from Theorem \ref{thm:uniform1} that $T_{\max} \leq T_{\mu}$ is uniformly bounded from above,
hence  $B:=1+T_{\mu}$ is under control.
Together with \eqref{eq: stability1} and \eqref{eq:clef bis} we thus obtain 
\begin{equation}\label{eq: stab2}
 {\mu }(\f<\phi-t) \leq \frac{C_2\delta }{|\chi|^m(-t)},
 \end{equation}
 where $\delta:=\left ( \int_X (\phi-\varphi) d\mu \right )^{\gamma}$.


   \smallskip
 
\noindent  {\it Conclusion}. 
Set $h(t)=-\chi(-t)$.
It follows from \eqref{eq: stab2} that for all $t\in [0,T_0]$,
$$
\frac{1}{g'(t)}={\mu }(\f<\phi-t) \leq \frac{C_2 \delta}{h^m(t)},
$$
hence
$$
h^m(t) \leq C_2 \delta g'(t)=(n+2c) C_2\delta  h''(t)  (h')^{n+2c-1}(t).
$$
Multiplying  by $h'$,  integrating between $0$ and $t$,
we infer that for all $t \in [0,T_0]$,
\begin{eqnarray*}
 h^{m+1}(t) 
&\leq&   (m+1) (n+2c) C_2\delta  \int_0^t  h''(s)  (h')^{n+2c}(s)ds \nonumber\\
&\leq &C_3 \delta \left ( (h')^{n+2c+1}(t)-1\right ),
\end{eqnarray*}
which yields
\begin{equation}
	 \label{eq: stab3}
	 1\leq  \frac{C_3 \delta (h')^{n+2c+1}(t)}{h^{m+1}(t)  + C_3 \delta}. 
\end{equation}
Recall that we have set  $m=n+ \e$ so that 
$$
\alpha:= m+1= n+\e +1> \beta:= n+2c +1.
$$
Raising both sides of \eqref{eq: stab3} to power $1/\beta$ we  obtain
\[
1 \leq  \frac{C_4\delta^{\frac{1}{\beta}}h'(t)}{(h(t)^{\alpha}+ C_3\delta)^{1/\beta}}. 
\]
We integrate between $0$ and $T_0$ and
 make the  change of variables $x= h(t)\delta^{-1/\alpha}$ to conclude
$
T_0 \leq C_5 \delta^{1/\alpha} \leq C_5 \left (\int_X (\phi-\varphi)_+ d\mu \right)^{\tau}$,
with $\tau=\gamma/\alpha$. Letting $T_0 \to T_{\max}$ we obtain the desired estimate. 
 \end{proof}

  \section{Refinements and extensions}
  
  We explain now how minor modifications of the proof  of Theorem \ref{thm:uniform1}
  provide other important uniform estimates in various contexts of K\"ahler geometry.

 \subsection{Big cohomology classes} \label{sec:big}

 Let $\theta$ be a smooth closed $(1,1)$-form  
 that represents a big cohomology class $\alpha$.
 We 
  set
 $$
 V_{\theta}(x):=\sup \{ v(x) , \; v \in \PSH(X,\theta)
 \text{ with } v \leq 0 \}.
 $$
  The latter is a $\theta$-psh function with minimal singularities, i.e. any other $\theta$-psh function
 $\f$ satisfies $\f \leq V_{\theta}+C$ for some constant $C$.  It is locally bounded in the ample locus
 ${\rm Amp}(\alpha)$, a Zariski open subset of $X$ where the cohomology class
 $\alpha$ behaves like a K\"ahler class.
 
 The Monge-Amp\`ere measure $(\theta+dd^c \f)^n$ of a $\theta$-psh function with minimal singularities
is well defined in ${\rm Amp}(\alpha)$,
and one can show that
it has finite mass independent from $\f$ and equal to 
$$
V_{\alpha}={\rm Vol}(\alpha)=\int_{{\rm Amp}(\alpha)} (\theta+dd^c V_{\theta})^n>0,
$$
 the volume of the class $\alpha$.

We refer the reader to \cite{BEGZ10} for more details on these notions and focus here on 
 slightly extending    \cite[Theorem B]{BEGZ10} by our new approach:
 
 \begin{theorem} \label{thm:big1}
 Let $\mu$ be a probability measure on $X$.
  If $\PSH(X,\theta) \subset L^m(\mu)$ for some $m>n$, 
 then there exists a unique    $\f \in \PSH(X,\theta)$  with minimal singularities
 such that $V_{\alpha}^{-1}(\theta+dd^c \f)^n=\mu$ and $\sup_X \f=0$.
Moreover
 $$
 ||\f-V_{\theta}||_{L^{\infty}(X)} \leq T_{\mu}
 $$
 for some uniform constant $T_{\mu}$.
 \end{theorem}

 \begin{proof}
 It follows from \cite[Theorem A]{BEGZ10} that there exists a unique finite energy solution $\f$.
 The key point for us here is to establish the a priori estimate.
Note that $\f \leq V_{\theta}$  since $\sup_X \f=0$
 Our goal is to show that $V_{\theta}-T_{max} \leq \f$, obtaining a uniform upper bound on $T_{max}$.

 A difficulty lies in the fact that $\theta$ is not a positive form. 
 We consider the positive current $\omega=\theta+dd^c V_{\theta}$ and set 
 $\tilde{\f}=\f-V_{\theta} \leq 0$. Observe that 
 $$
 \theta_{\f}:=\theta+dd^c \f=\omega+dd^c \tilde{\f}=:\omega_{\tilde{\f}} \geq 0.
 $$
 Our plan is thus to show that the "$\omega$-psh" function $\tilde{\f}$ is bounded.
 
 As in the proof of Theorem \ref{thm:uniform1} we let $\chi: \R^- \rightarrow \R^-$ denote a concave increasing function
 such that $\chi(0)=0$ and $\chi'(0)= 1$.
 We set $\p=V_{\theta}+\chi \circ \tilde{\f}$ and consider
 $$
 u=P_{\theta}(\p)=P_{\theta}(V_{\theta}+\chi \circ (\f-V_{\theta})) .
 $$
Observe that 
 \begin{eqnarray*}
\theta+dd^c \p &= & \chi' \circ \tilde{\f} \, \theta_{{\f}}+[1-\chi' \circ \tilde{\f}] \omega+
\chi'' \circ \tilde{\f} \, d \tilde{\f} \wedge d^c \tilde{\f} \\
&\leq & \chi' \circ (\f-V_{\theta}) \,  (\theta+dd^c \f).
\end{eqnarray*}

The envelopes in the context of big cohomology classes enjoy similar properties as those reviewed
in Section \ref{sec:envelopes}.
In particular the complex Monge-Amp\`ere measure
$(\theta+dd^c P_{\theta}(\p))^n$ is concentrated on 
the contact set $\{P_{\theta}(\p)=\p\}$ (see \cite[Theorem 2.7]{GLZ19})
and the big-version of Lemma \ref{lem:GLkey} holds, showing that
 $$
V_{\alpha}^{-1} (\theta+dd^c u)^n  
 \leq 1_{\{P_{\theta}(\p)=\p\}} (\chi' \circ (\f-V_{\theta}))^n \mu.
 $$
  The rest of the proof is identical to that of Theorem \ref{thm:uniform1}.
 \end{proof}

 \subsection{Degenerating families} \label{sec:families}
  
  Families of K\"ahler-Einstein varieties have been intensively studied in the past decade,
 requiring one to analyze the associated family of complex Monge-Amp\`ere equations.
 We refer the reader to \cite{Tos09,Tos10,ST12,GTZ13,DnGG20,Li20} for 
 detailed examples and  
 geometrical motivations.
 
 The most delicate situation is when the volume of the fiber collapses.
Theorem \ref{thm:uniform1} yields a uniform bound in this case, providing
an alternative proof and an extension of the main results of \cite{EGZ08,DP10}:

\begin{coro}
Let 
$\omega_t$ be a family of semi-positive and big forms on $X$, 
and assume there is a fixed   form $\Theta$ such that $0 \leq \omega_t \leq \Theta$.
We let $V_t:=\int_X \omega_t^n>0$ denote the volume of $(X,\omega_t)$.
 Let $\mu$ be a probability measure.
 If $\PSH(X,\Theta) \subset L^m(\mu)$ for some $m>n$, then
any solution $\f_t \in \PSH(X,\omega_t) \cap L^{\infty}(X)$  to 
 $$
\frac{1}{V_t} (\omega_t+dd^c \f_t)^n=  \mu 
 $$
  satisfies
 $
 {\rm Osc}_X (\f_t) \leq  T_{\mu}
 $
 for some uniform constant $T_{\mu}$.
\end{coro}

The point here is that the estimate is uniform in $t$ although the volumes
$V_t$ may degenerate to zero (volume collapsing).

\begin{proof}
 Theorem \ref{thm:uniform1}  provides a uniform bound
 $ {\rm Osc}_X (\f_t) \leq  T(A_m(\omega_t,\mu))$, where
  $$
A_m(\omega_t,{\mu}):= \sup \left\{ \int_X (-\p)^m d\mu, \; \p \in\PSH(X,\omega_t) \text{ with } \sup_X \p=0 \right\}.
 $$
Since $\PSH(X,\omega_t) \subset\PSH(X,\Theta)$ and 
$\PSH(X,\Theta) \subset L^m(\mu)$, we obtain that
$A_m(\omega_t,{\mu}) \leq A_m(\Theta,{\mu})<+\infty$.
The uniform upper bound follows.
\end{proof}

 This uniform estimate shows in particular that in many geometrical contexts,
 a uniform control on the $L^{n+\e}$-norm of the Monge-Amp\`ere
 potentials $\f_t$ suffices to obtain a $L^{\infty}$-control of the latter.
 
One can   obtain similarly  uniform estimates when the underlying 
complex structure is also changing:
let ${\mathcal X}$ be an irreducible and reduced complex K\"ahler space,
and let $\pi:{\mathcal X} \rightarrow \D$
denote a proper, surjective holomorphic map such that each fiber 
$X_t=\pi^{-1}(t)$ is an
$n$-dimensional, reduced, irreducible, compact K\"ahler space, for any $t \in \D$.
Given $\omega$ a K\"ahler form on ${\mathcal X}$ and $\omega_t:=\omega_{|X_t}$,
one can consider the complex Monge-Amp\`ere equations
$$
\frac{1}{V}(\omega_t+dd^c \f_t)^n=\mu_t,
$$
where 
\begin{itemize}
\item the volume $V=\int_{X_t} \omega_t^n$ turns out to be independent of $t$, and
\item $\mu_t$ is a family of probability measures on each fiber $X_t$
(e.g. the normalized Calabi-Yau measures of a degenerating
family of Calabi-Yau manifolds).
\end{itemize}
In many concrete geometrical situations (see e.g. \cite{GTZ13,DnGG20,Li20}), one can check that
$A_m(\omega_t,\mu_t) \leq A$ is uniformly bounded from above
for some $m>n$ (often any $m>1$).
If one can further uniformly compare $\sup_{X_t} \f_t$ and $ \int_{X_t} \f_t \frac{\omega_t^n}{V}$,
 then Theorem \ref{thm:uniform1} then applies and provides a uniform $L^{\infty}$-estimate.
It is thus sometimes not necessary to establish
a  uniform Skoda integrability theorem in families
(compare with \cite{DnGG20,Li20b}).

\subsection{Relative a priori $L^{\infty}$-bounds} \label{sec:relative}
 Fix $\omega$  a   semi-positive and big $(1,1)$ form, and 
$\rho$ an $\omega$-psh function with analytic singularities such that
$\omega+dd^c \rho \geq \delta \omega_X$ is a K\"ahler current
which is smooth in the ample locus ${\rm Amp}(\omega)$.
We normalize $\rho$ so that $\sup_X \rho=0$ and set $V=\int_X \omega^n>0$.

 We consider in this section the degenerate complex Monge-Amp\`ere equation 
 \begin{equation} \label{eq:geom}
V^{-1}  (\omega+dd^c \f)^n=  \mu= fdV_X,
 \end{equation} 
where  $\mu$ is a probability measure whose density $f \in L^1(X)$ does not belong to any  
good Orlicz class (see section \ref{sec:abscont}).
Since $\mu$ does not charge pluripolar sets, there exists a unique
"finite energy solution" $\f \in {\mathcal E}(X,\omega)$ (see \cite{GZbook}),
but one cannot expect  anylonger that $\f$ is globally bounded on $X$.

Given $\p$ a quasi-plurisubharmonic function on $X$ and $c>0$, we set
$$
E_c(\p):=\{ x \in X , \; \nu(\p,x) \geq c \},
$$
where $\nu(\p,x)$ denotes the  Lelong number of $\p$ at $x$. 
A celebrated theorem of Siu ensures that for any $c>0$, the
set $E_c(\p)$ is a closed analytic subset of $X$.

\begin{thm}   \label{thm:relative}
Assume   $f=g e^{-\p}$, where $0 \leq g \in L^p(dV_X)$, $p>1$,  and  $\p$ is a quasi-psh function.
Then  there exists a unique $\f \in\mathcal{E}(X,\omega)$ such that
\begin{itemize}
\item $\alpha (\p+\rho)-\beta \leq \f \leq 0$ with $\sup_X \f=0$;
\item $\f$ is locally bounded in the Zariski open set $\Omega:= {\rm Amp}(\omega) \setminus  E_{\frac{1}{q}}(\p)$;
\item $ V^{-1}(\omega+dd^c \f)^n=   fdV_X 
 \; \;  \text{ in } \; \;
\Omega$,
\end{itemize}
where $\alpha,\beta>0$ depend on  an upper bound for $||g||_{L^p}$
and $\frac{1}{p}+\frac{1}{q}=1$.
\end{thm}

  When   $f \leq e^{-\p}$ for some quasi-psh function $\p$,
it has been shown by DiNezza-Lu 
 \cite[Theorem 2]{DnL17} that the normalized solution $\f$ to \eqref{eq:geom} is locally bounded in
the complement of the   set $\{\p=-\infty\}$.
The proof of DiNezza-Lu is a generalization of the method of Ko{\l}odziej \cite{Kol98} that makes use of 
a theory of generalized Monge-Amp\`ere capacities further developed in \cite{DnL15}.
 We slightly extend this result here and propose  
a brand new proof  
using envelopes and Theorem \ref{thm:uniform1}.

  \begin{proof}
  {\it Reduction to analytic singularities.}
We let $q$ denote the conjugate exponent of $p$,  set $r=\frac{2p}{p+1}$,
and note that $1 < r< p$.
If the Lelong numbers of $\p$ are all less than $\frac{1}{q}$, it follows from H\"older inequality that
$f \in L^r(dV_X)$, since
$$
\int_X f^r dV_X =\int_X g^r e^{-r \p} dV_X
\leq \left( \int_X g^p dV_X \right)^{\frac{r}{p}} \cdot 
\left( \int_X e^{-\frac{pr}{p-r} \p} dV_X \right)^{\frac{p-r}{p}} ,
$$
  where the last integral is finite by Skoda's integrability theorem  \cite[Theorem 8.11]{GZbook}
 if  $\frac{pr}{p-r} \nu(\p,x)<2$
 for all $x \in X$, which is equivalent to $\nu(\p,x)<\frac{1}{q}$.
 
 It is thus natural to expect that the solution $\f$ will be locally bounded in the complement 
 of the closed analytic set $E_{q^{-1}}(\p)$. 
 It follows from Demailly's equisingular approximation technique (see \cite{Dem15}) that
 there exists a sequence $(\p_m)$ of quasi-psh functions on $X$ such that
 \begin{itemize}
 \item $\p_m \geq \p$ and $\p_m \rightarrow \p$ (pointwise and in $L^1$);
 \item $\p_m$ has analytic singularities concentrated along $E_{m^{-1}}(\p)$;
 \item $dd^c \p_m \geq -K \omega_X$, for some uniform constant $K>0$;
 \item $\int_X e^{2m(\p_m-\p)} dV_X <+\infty$ for all $m$.
 \end{itemize}
  We choose $m=[q]$, set
  $g_m:=g e^{\p_m-\p}$, and observe that
 \begin{eqnarray*}
 \int_X g_m^r  & \leq  &
 \left( \int_X e^{2m(\p_m-\p)}dV_X \right)^{\frac{1}{2m}} \cdot 
\left( \int_X   g_m^{\frac{2mr}{2m-r}}  dV_X \right)^{\frac{2m-r}{2m}}  \\
& \leq & \left( \int_X e^{2m(\p_m-\p)}dV_X \right)^{\frac{1}{2m}} \cdot 
\left( \int_X   g_m^{p}  dV_X \right)^{\frac{2m-r}{2m}} <+\infty
 \end{eqnarray*}
  if we choose $r^{-1}=p^{-1}+(2m)^{-1}<1$ so that $\frac{2mr}{2m-r}=p$.
 By replacing $\p$ by $\p_{[q]} \geq \p$ and $g$ by $g_m \in L^r$
     in the sequel, we can thus  
  assume that 
  \begin{itemize}
  \item   $\p$  has analytic singularities and
   is smooth in $X \setminus E_{q^{-1}}(\p)$;
   \item the functions $\tilde{\p}:=a\p+\rho$ is $\omega$-psh, with $a:=\delta/K$.
  \end{itemize}

 \medskip

 \noindent  {\it Uniform integrability of $\f$}.
It is a standard measure theoretic fact that the density $f$ belongs to an Orlicz class
$L^w$ for some convex increasing weight $w:\R^+ \rightarrow \R^+$
such that $w(t)/t \rightarrow +\infty$ as $t \rightarrow +\infty$.
Set $\chi_1(t):=-(w^*)^{-1}(-t)$, where $w^*$ denotes the Legendre transform of $w$.
Thus $\chi_1:\R^- \rightarrow \R^-$  is a convex increasing weight 
such that $\chi_1(-\infty)=-\infty$ and
$$
\int_X (-\chi_1 \circ \f) (\omega+dd^c \f)^n 
\leq \int_X w \circ f dV_X+\int_X (-\f) dV_X \leq C_0,
$$
as follows from the additive version of H\"older-Young inequality
and the compactness of $\sup$-normalized $\omega$-psh functions.

It follows  that $\f$ belongs to a compact subset of the finite 
energy class ${\mathcal E}_{\chi_1}(X,\omega)$, hence
for all $\lambda \in \R$,
\begin{equation} \label{eq:SKodaDnL}
\int_X \exp(- \lambda \f) dV_X \leq C_{\lambda},
\end{equation}
for some $C_{\lambda}$ independent of $\f$ 
(see \cite{GZ07,GZbook} for more information).

    \smallskip
 
   \noindent  {\it The envelope construction}.
   Let $u=P(2\f-\tilde{\p})$ denote the greatest $\omega$-psh function that lies below $2\f-\tilde{\p}$.
   Since $h=2\f-\tilde{\p}$ is bounded from below and quasi-continuous, it follows from Proposition \ref{pro:orthog} that
  the measure 
   $(\omega+dd^c u)^n$ is supported on the contact set 
   ${\mathcal C}=\{ u=2\f-\tilde{\p}\}$. Thus
  $$
   (\omega+dd^c u)^n 
  \leq  1_{{\mathcal C}}  (\omega+dd^c (2\f-\tilde{\p}))^n   
\leq   1_{{\mathcal C}}  (2\omega+dd^c (2 \f))^n.
$$
Since $v \leq w$ on $X$, it follows from Lemma \ref{lem:Demkey} that
    \begin{equation} \label{eq:comp}
    1_{\{v=w\}} (2\omega+dd^c v)^n \leq  1_{\{v=w\}} (2\omega+dd^c w)^n,
    \end{equation}
where
   \begin{itemize}
   \item $v=u+\tilde{\p}$ is $2\omega$-psh and $u +\tilde{\p} \leq 2\f=w$ on $X$;
   \item $\{ u+\tilde{\p}=2 \f \}$ coincides with the contact set ${\mathcal C}$.
   \end{itemize}
 Therefore, it follows from \eqref{eq:comp} that
 \begin{eqnarray*}
  1_{{\mathcal C}}  (2\omega+dd^c (u+\tilde{\p}))^n 
   & \leq  & 1_{{\mathcal C}}  (2\omega+dd^c (2\f))^n  \\
  &\leq&  1_{{\mathcal C}}   2^n c g e^{-\p} dV_X \\
& \leq & 1_{{\mathcal C}}   2^n c g e^{u/a} e^{-2\f/a} dV_X
\leq c_1 g e^{-2\f/a} dV_X,
 \end{eqnarray*}
 using that $\sup_X u \leq c_2$ is uniformly bounded from above, as we explain below.
 
It follows from H\"older inequality and \eqref{eq:SKodaDnL} that the measure $g e^{-2\f/a} dV_X$ satisfies
the assumption of Theorem \ref{thm:uniform1}.
 We infer that $u\geq -M$ is uniformly bounded below, hence
  $$
  2\f = (2\f-\tilde{\p}) +\tilde{\p} \geq u+\tilde{\p} \geq \frac{\delta}{K}\p+\rho-M.
  $$
  The desired a priori estimate follows with $\beta=M/2$ and $\alpha=\max(1,\delta/2K)$.
   
  \smallskip
 
   \noindent  {\it Bounding $\sup_X u$ from above}.
   We can assume without loss of generality that $\sup_X \tilde{\p}=0$.
   Consider $G=\{\tilde{\p}>-1\}$; this is a non empty plurifine open set.
   Observe that for all $x \in G$,
   $u(x) \leq (2\f-\tilde{\p})(x) \leq 1$, hence
   $$
   u(x) -1 \leq V_{G,\omega}(x):=\sup \{ w(x) , \, w \in\PSH(X,\omega) \text{ with } w \leq 0 \text { on } G \}.
   $$
It follows from \cite[Theorem 9.17.1]{GZbook}   
 that $\sup_X V_{G,\omega}=C$ is finite since $G$ is non-pluripolar, thus
 $\sup_X u \leq c_2=1 +\sup_X V_{G,\omega}=1+C$.
   \end{proof}

   \section{The local context} \label{sec:local}

   \subsection{Cegrell classes}
   
    We fix $\Omega \subset \C^n$  a bounded hyperconvex domain, i.e. there exists a continuous  plurisubharmonic
     function $\rho: \Omega \rightarrow [-1,0)$ whose sublevel sets $\{\rho <-c\}\Subset \Omega$ 
     are relatively compact for all $c>0$. 

Let ${\mathcal T}(\Omega)$ denote the 
set of bounded 
plurisubharmonic 
functions $u$ in $\Omega$ such that $\lim _{z\to \zeta} u (z) = 0,$ 
for every $\zeta \in \partial \Omega,$ and $\int_\Omega(dd^c u )^n <+\infty$.
Cegrell \cite{Ceg98, Ceg04} has 
studied the
complex Monge-Amp\`ere operator $(dd^c \cdot)^n$ and 
 introduced different
classes of 
plurisubharmonic functions on which the latter is well defined: 
\begin{itemize}
\item  ${\rm DMA}(\Omega)$ is the set of psh functions
 $u$ such that for all $z_0 \in \Omega$, there exists a neighborhood $V_{z_0}$ of
$z_0$ and $u_j \in {\mathcal T}(\Omega)$ a decreasing sequence which
converges to  $u$ in $V_{z_0}$ and satisfies
$\sup_j \int_{\Omega} (dd^c u_j)^n <+\infty$.

\item 
a function $u$ belongs to ${\mathcal F}(\Omega)$ iff  there exists $u_j \in {\mathcal T}(\Omega)$
a sequence decreasing towards $u$ {in all of } $\Omega$, which satisfies
$\sup_j \int_{\Omega} (dd^c u_j)^n<+\infty$;

\item a function $u$ belongs to   ${\mathcal E}^p(\Omega)$ 
if there exists a sequence of
functions $u_j \in {\mathcal T}(\Omega)$ decreasing towards $u$ in   $\Omega$
with  $\sup_j \int_{\Omega} (-u_j)^p (dd^c u_j)^n<+\infty$. 
\item a function $u$ belongs to   ${\mathcal F}^p(\Omega)$ 
if there exists a sequence of
functions $u_j \in {\mathcal T}(\Omega)$ decreasing towards $u$ in   $\Omega$
with   $\sup_j \int_{\Omega} [1+(-u_j)^p] (dd^c u_j)^n<+\infty$.
\end{itemize}

Given $u \in \mathcal{E}^p(\Omega)$ we define the weighted energy of $u$ by 
\[
E_p(u):= \int_{\Omega} (-u)^p (dd^c u)^n<+\infty.
\]  
The operator $(dd^c \cdot )^n$ is well defined on these sets, and continuous under decreasing limits. If $u\in \mathcal{E}^p(\Omega)$ for some $p>0$ then $(dd^c u)^n$ vanishes on all pluripolar sets \cite[Theorem 2.1]{BGZ09}. If $u\in \mathcal{E}^p(\Omega)$ and $\int_{\Omega} (dd^c u)^n <+\infty$ then $u\in \mathcal{F}^p(\Omega)$. Also, note that 
$$
{\mathcal T}(\Omega) \subset {\mathcal F}^p(\Omega) \subset{\mathcal F}(\Omega)\subset {\rm DMA}(\Omega)
\; \; \text{ and } \; \; 
{\mathcal T}(\Omega) \subset {\mathcal E}^p(\Omega)\subset {\rm DMA}(\Omega). 
 $$
 
Cegrell has characterized the range of the complex Monge-Amp\`ere operator acting on the 
classes  ${\mathcal E}^p(\Omega)$:
 
  \begin{theorem} \label{thm:cegrell}
  \cite[Theorem 5.1]{Ceg98}
   	Let $\mu$ be a probability measure in $\Omega$.
   	 There exists a function $u\in {\mathcal F}^p(\Omega)$ such that $(dd^c u)^n = \mu$
   	 if and only if $ {\mathcal F}^p(\Omega) \subset L^p(\Omega)$.
   \end{theorem}
   
   A simplified variational proof of this result has been provided in \cite{ACC12}.

 \subsection{Dirichlet problem}

We have the following local analogue of Theorem \ref{thm:uniform1}:   
   
   \begin{theorem}
   	Assume $\mu$ is a probability measure in $\Omega$ 
   	 and $\mathcal{F}(\Omega) \subset L^{m}(\mu)$, for some $m>n$. 
   	 Then there exists a unique bounded function $u\in \mathcal{F}(\Omega)$ such that $(dd^c u)^n = \mu$. 
   	 The upper bound on $\sup_{\Omega}|u|$ only depends  on $A_m(\mu),m,n$, where 
   	\[
   	A_m(\mu) := \sup \left \{\int_{\Omega} (-u)^m d\mu \; :\; u \in \mathcal{T}(\Omega)
   	\text{ with } \int_{\Omega} (dd^c u)^n \leq 1 \right \}.
   	\] 
   \end{theorem}

   \begin{proof}
   We first explain why the integrability condition $\mathcal{F}(\Omega) \subset L^m(\Omega,\mu)$ is equivalent to the finiteness of $A_m$. Indeed, if $A_m$ is not finite then there exists a sequence $(u_j)$ in $\mathcal{T}(\Omega)$ such that $\int_{\Omega}(dd^c u_j)^n \leq 1$ but $\int_{\Omega} |u_j|^m d\mu \geq 4^{jm}$. Let $u:= \sum_{j=1}^{+\infty} 2^{-j} u_j$. Then, by \cite[Corollary 5.6]{Ceg04}, we have  $u\in \mathcal{F}(\Omega)$, but 
   \[
   \int_{\Omega} (-u)^m d\mu\geq 2^{-jm} \int_{\Omega} (-u_j)^m d\mu \geq 2^{jm} \to +\infty. 
   \]

   	
   It follows from 	Theorem \ref{thm:cegrell} that
   	there exists $\varphi \in \mathcal{F}(\Omega)$ such that $(dd^c \varphi)^n =\mu$.  We assume for the moment that $u\in \mathcal{T}$ is bounded and we establish a uniform bound for $\varphi$. Set 
 \[
 T_{\max}:= \sup \{t>0 \; : \; \mu(\f <-t) >0\}. 
 \]
 Our goal is to establish a precise bound on $T_{max}$.
    By definition,  $-T_{max} \leq \f$ almost everywhere with respect to $\mu$, hence $(dd^c \max(\varphi,-T_{\max}))^n \geq (dd^c \f)^n$ and the domination principle, \cite[Corollary 3.31]{GZbook}, gives $\varphi\geq -T_{\max}$, providing the desired a priori bound
    $|\f| \leq T_{max}$.

We let $\chi:  \mathbb{R}^- \rightarrow \R^-$ denote a {\it concave} increasing function
 such that $\chi(0)=0$ and $\chi'(0) = 1$. 
 We set $\p=\chi \circ \f$, $u=P(\p) \in \mathcal{T}(\Omega)$ the largest psh function in $\Omega$ which lies below $\psi$,  and observe that
\begin{eqnarray*}
dd^c \p &= & \chi' \circ \f \omega_\f+\chi'' \circ \f d\f \wedge d^c \f \leq  \chi' \circ \f dd^c \f.
\end{eqnarray*}
Since $\p \geq \chi'(-T_{\max}) \varphi$ and the latter is in $\mathcal{T}(\Omega)$ we deduce that $u\geq \chi'(-T_{\max}) \varphi$ and $u\in \mathcal{T}(\Omega)$.

 Although the function $\p$ is not psh, this provides a bound from above on the positivity of $dd^c \p$ which allows
 to control the Monge-Amp\`ere of its envelope, see \cite[Lemma 4.1 and Lemma 4.2]{DnGL20}, 
 $$
 (dd^c u)^n \leq {\bf 1}_{\{u=\p\}} (dd^c \p)^n 
 \leq  (\chi' \circ \f)^n \mu.
 $$
The above inequalities hold for smooth functions and the general case of bounded psh functions can be obtained as in the proof of Lemma \ref{lem:GLkey}.  

We thus get a uniform control on the Monge-Amp\`ere mass of $u$: 
 \begin{eqnarray*}
 \int_{\Omega}(dd^c u)^n 
 &\leq &  \int_{\Omega} (\chi' \circ \f)^n d\mu.
 \end{eqnarray*}
 We are going to choose below the weight $\chi$ in such a way that
$\int_{\Omega} ( \chi' \circ \f)^{n} d\mu =B \leq 2$ is a finite constant under control.
This provides a uniform upper bound on
 $||u||_{L^m(\mu)}$. Using Chebyshev inequality we thus obtain
\begin{equation} \label{eq:clefter}
 {\mu }(\f<-t) \leq \frac{\int_{\Omega} |\chi(\varphi)|^m d\mu}{|\chi|^m(-t)} \leq \frac{\int_{\Omega} |u|^m d\mu }{|\chi|^m(-t)} \leq  \frac{A_m}{|\chi|^m(-t)},
\end{equation}
where $A_m \geq 1$ is an upper bound for $\int_{\Omega} |u|^m d\mu$.

   \smallskip
 
\noindent  {\it Choice of $\chi$}.
 We use again Lebesgue's formula:   if $g:  \mathbb{R}^+ \rightarrow \R^+$  is  increasing and normalized by $g(0)=1$ then
$$
\int_{\Omega} g \circ (-\f) d\mu = \mu(\Omega) + \int_0^{T_{\max}} g'(t) {\mu }(\f<-t) dt.
$$
Setting $g(t)=[\chi' (-t)]^{n}$  we define $\chi$ by imposing $\chi(0)=0$, $\chi'(0)=1$, and 
$$
g'(t)=  \begin{cases}
	\dfrac{1}{(1+t)^2 {\mu }(\f<-t)}, \; \text{if}\;  t \in [0,T_0], \\
	\;\\
	\dfrac{1}{t^2+1}, \; \text{if}\;  t >T_0.
\end{cases}	
 $$
This choice guarantees that 
$$
\int_{\Omega}( \chi' \circ \f)^{n} d\mu \leq  \mu(\Omega)+ \int_0^{+\infty} \frac{dt}{(1+t)^2} =2.
$$

 \noindent  {\it Conclusion}.  We set $h(t)=-\chi(-t)$ and work with the positive counterpart of $\chi$. Note that $h(0)=0$ and $h'(t)=[g(t)]^{\frac{1}{n}}$ is positive increasing, hence
$h$ is convex increasing (so $\chi$ is concave increasing and negative).

Together with \eqref{eq:clefter} our choice of $\chi$ yields, for all $t\in [0,T_0]$, 
$$
\frac{1}{(1+t)^2g'(t)}={\mu }(\f<-t) \leq \frac{A_m}{h^m(t)}.
$$
This reads
$$
h^m(t) \leq A_m (1+t)^2 g'(t)=n A_m (1+t)^2 h''(t)  (h')^{n-1}(t).
$$
We integrate this inequality as in the proof of Theorem \ref{thm:uniform1}
and  obtain 
\[
T_0\leq C', 
\] 
for some uniform constant $C'$ depending on $n,m,A_m$.  

\smallskip

To finish the proof we write $\mu = f (dd^c \phi)^n$, where $0\leq f\in L^1(\Omega, (dd^c \phi)^n)$ and $\phi\in \mathcal{T}(\Omega)$. This is known as Cegrell's decomposition theorem \cite[Theorem 6.3]{Ceg98}. We next solve $(dd^c \f_j)^n = \min(f,j) (dd^c \phi)^n$ with $\f_j \in \mathcal{T}(\Omega)$. Since $(dd^c \f_j)^n \leq \mu$, our estimate above shows that $|\f_j| \leq C$ for a uniform constant $C$. The comparison principle also gives that $\f_j$ is decreasing and $\varphi\leq \varphi_j$, thus $u:=\lim_j \f_j \in \mathcal{F}(\Omega)$ is bounded and $(dd^c u)^n =\mu$. It  then follows from \cite[Theorem 5.15]{Ceg04} that $u=\varphi$, finishing the proof. 
   \end{proof}


\end{document}